\documentclass[11pt]{article}
\usepackage{amsmath,amssymb,amsthm}
\usepackage[colorlinks,citecolor=blue]{hyperref}
\usepackage[utf8]{inputenc}
\usepackage{tikz}
\usetikzlibrary{shapes.geometric,calc}
\usepackage[all]{xy}
\definecolor{darkblue}{rgb}{0.0,0,0.7} 
\newcommand{\darkblue}{\color{darkblue}} 
\definecolor{darkred}{rgb}{0.7,0,0} 
\usepackage{hyperref}
\usepackage{cleveref}
\usepackage{quiver}
\usepackage{adjustbox}

\newtheorem{theorem}{Theorem}[section] 
\newtheorem{proposition}[theorem]{Proposition} 
 
\newtheorem{corollary}[theorem]{Corollary} 
\newtheorem{lemma}[theorem]{Lemma}

\theoremstyle{definition}
\newtheorem{definition}[theorem]{Definition}
\newtheorem{remark}[theorem]{Remark} 
\newtheorem{example}[theorem]{Example}

\newcommand{\Tc}{\mathcal{T}}

\newcommand{\Ima}{\operatorname{Im}}

\newcommand{\kk}{\mathbf{k}}

\newcommand{\F}{\mathcal{F}}

\newcommand{\Tam}{\operatorname{Tam}}

\newcommand{\modf}{\operatorname{mod}}

\newcommand{\Hom}{\operatorname{Hom}}

\newcommand{\add}{\operatorname{add}}
\newcommand{\Ext}{\operatorname{Ext}}

\newcommand{\Filt}{\operatorname{Filt}}

\newcommand{\T}{\mathcal{T}}

\newcommand{\ideal}{\operatorname{Ideal}}
\newcommand{\Int}{\operatorname{Int}}
\newcommand{\Fc}{\mathcal{F}}
\newcommand{\Fb}{\operatorname{F}}
\newcommand{\Tb}{\operatorname{T}}

\crefname{figure}{figure}{figures}
\Crefname{figure}{Figure}{Figures}

\newcommand{\emp}[1]{\emph{\darkblue #1}} 

\newcommand{\darkred}{\color{darkred}} 
\newcommand{\defn}[1]{\emph{\darkred #1}}

\title{A remark on s-torsion pairs and on the lattice of Dyck paths}
\author{Baptiste Rognerud}
\date{}

\usepackage{array}
\newcolumntype{P}[1]{>{\centering\arraybackslash}p{#1}}

\begin{document}

\maketitle


\begin{abstract}
There are three classical lattices on the Catalan numbers: the Tamari lattice, the lattice of noncrossing partitions and the lattice of Dyck paths. The first is known to be isomorphic to the lattice of torsion classes of the path algebra of an equioriented quiver of type $A$ and the second is known to be isomorphic to its lattice of wide subcategories. Inspired by the notion of $s$-torsion classes of Adachi, Enomoto and Tsukamoto, in this short note we interpret the lattice of Dyck paths as a lattice of subcategories. 
 \end{abstract}
\section{Introduction}
There are three classical lattices on the Catalan numbers: the Tamari lattice, the lattice of noncrossing partitions and the lattice of Dyck paths. These three lattices have many Coxeter/Dynkin interpretations. For example, the Tamari lattice appears as a \emp{cambrian} lattice of type $A$ in the sense of \cite{reading2006cambrian} and the classical noncrossing partitions are known to be the type $A$ (equioriented) of the noncrossing partitions associated to a Coxeter group. The lattice of Dyck paths is isomorphic to the lattice of \emp{nonnesting partitions} which are the order ideals in the poset of positive roots of a \emp{root system} of type $A$, see for example \cite[Section 4.2]{gobet2016noncrossing}. 

The first two lattices can also be viewed as lattices of particular \emp{subcategories satisfying a closure property}, ordered by inclusion. This kind of realizations, among other things, makes it possible to obtain clean proofs of the lattice property. Moreover, it allows us to use representation theory to study the lattice. The Tamari lattice is known to be isomorphic to the lattice of \emp{torsion classes} of  the path algebra of an equioriented quiver of type $A$, see for example \cite{thomas_tamari}. By \cite[Theorem 1]{ingalls2009noncrossing}, the lattice of (classical) noncrossing partitions is isomorphic to the lattice of \emp{wide subcategories} of the path algebra of the same quiver. However, as far as the author knows, no such interpretation is known for the lattice of Dyck paths.

The lattice of Dyck path is in some sense much nicer than the two others: the ordering relation is particularly simple and it is a \emp{distributive lattice}. However, its distributivity is an \emp{obstruction} to its realization as the lattice of torsion classes of a finite dimensional algebra. Indeed, if an algebra has two simple modules $S$ and $S'$ with a non split extension, they generate a penthagonal sublattice of the lattice of torsion classes, hence the only algebras which have a distributive lattice of torsion classes are the semisimple algebras (and the products of local algebras see \cite{luo2023boolean} for more details). In that case, the lattice of torsion classes is a boolean lattice. 

Inspired by the notion of s-torsion pairs of \cite{aet}, we introduce the notion of \defn{$\omega_n$-torsion pairs} for the category of finitely generated modules over an artinian algebra (see \cref{def:main}).  When the algebra has finite global dimension $n$, we recover \cite[Example 3.3]{aet}, however $\omega_n$-torsion pairs can be defined without any extra assumption on the artinian algebra, in particular we will consider the $\omega$-torsion pairs (= $\omega_1$-torsion pairs) of the incidence algebra of finite posets. We prove that the set of $\omega_n$-torsion pairs is a sublattice of the lattice of torsion pairs, therefore it is a \emp{semidistributive lattice}. In the case $n=1$, it is a always a \emp{finite distributive} lattice. Moreover we prove
\begin{proposition}
Let $(P,\leq)$ be a finite poset and $A$ its incidence algebra over a field $\kk$. The poset of $\omega$-torsion pairs of $A$ is isomorphic to the distributive lattice of order ideals of $(P,\leq)^{op}$. 
\end{proposition}

By the famous Birkhoff representation theorem any finite distributive lattice can be realized as the lattice of order ideals of a given finite poset (the poset of join-irreducible elements). So any finite distributive lattice can be realized as a \emp{natural sublattice of a lattice of torsion pairs} of an incidence algebra.  For the lattice of Dyck paths, we obtain the following result, where $(\Int([n-1]),\subseteq)$ is the set of intervals of a total order with $n-1$ elements, ordered by containment. 

\begin{theorem}\label{thm:main}
The lattice of Dyck paths of length $n$ is isomorphic to the lattice of $\omega$-torsion pairs of the incidence algebra of $(\Int([n-1]),\subseteq)^{op}$. 
\end{theorem}
To conclude this note, we would like to point out that there is another way to give a \emp{torsion flavor} to the Dyck paths. The lattice of torsion pairs of the path algebra of an equioriented quiver of type $A_{n-1}$ is isomorphic to the \emp{Tamari lattice} $\Tam_n$ on the binary trees with $n$ inner vertices. We can therefore use the following result, which can be found in \cite{geyer}, even if it is not written explicitly in these terms.

\begin{theorem}[Geyer]\label{thm:2}
The lattice of Dyck paths is isomorphic to the lattice of congruences of the Tamari lattice.
\end{theorem}
We believe that this result should be more widely known, and we present two simple proofs, the first as a corollary of \cite{DIRRT} and the second as a corollary of \cite{luo2024lattice}. 

\section{$\omega$-torsion pairs}
Let $A$ be a finite dimensional $\kk$-algebra or more generally an artinian $\kk$-algebra. We denote by $\modf A$ the category of finitely generated right $A$-modules and by a subcategory, we always mean a \emp{full subcategory} which is \emp{closed under isomorphisms}. Let us briefly recall that a torsion pair of $\modf A$ is a pair $(\Tc,\Fc)$ of subcategories such that 
\begin{enumerate}
\item For every $T\in \Tc$ and $F\in \Fc$, we have $\Hom_A(T,F) = 0$.
\item For every $X\in \modf A$, there exists a short exact sequence $0\to t(X) \to X \to f(X)\to 0$ with $t(X) \in \Tc$ and $f(X)\in \Fc$. 
\end{enumerate}
If $(\Tc,\Fc)$ is a torsion pair, we say that $\Tc$ is a \defn{torsion class} and $\Fc$ a \defn{torsion-free} class. If $\Tc$ is a torsion class, it is closed under \emp{extensions} and \emp{quotients}. Dually a torsion-free class is closed under \emp{extensions} and \emp{subobjects}. A subcategory which is closed under extensions, quotients and subobjects is called a \defn{Serre subcategory}. Moreover if $\T$ is a subcategory closed under extensions and quotients, then $(\T,\T^{\bot})$ is the unique torsion pair with $\T$ as torsion class. Here $\T^{\bot} = \{ X\in \modf A\ |\ \Hom_{A}(\T,X) = 0 \}$ denotes the right hom-orthogonal to $\T$. 

Set theoretical issues aside, the torsion pairs form a (complete) lattice for the partial ordering $(\Tc,\Fc) \preceq (\Tc',\Fc')$ if $\Tc \subseteq \Tc'$ and $\Fc'\subseteq \Fc$. Moreover, the meets and the joins of this lattice are easily computed:
\[
(\Tc,\Fc) \land (\Tc',\Fc') = (\Tc \cap \Tc',\operatorname{F}(\Fc\cup \Fc')), 
\]
\[
(\Tc,\Fc) \lor (\Tc',\Fc') = (\operatorname{T}(\Tc \cup \Tc'),\Fc\cap \Fc'), 
\]
where $\Tb(S)$ is the smallest torsion class containing $S$ and $\Fb(S)$ is the smallest torsion-free class containing $S$, for $S\subseteq \modf A$.

\begin{definition}\label{def:main}
A pair a subcategories $(\T,\F)$ of $\modf A$ is an \defn{$\omega_n$-torsion pair} if 
\begin{enumerate}
\item $(\T,\F)$ is a torsion pair of $\modf A$. 
\item $\Ext^{n}_A(\T,\F) = 0$.
\end{enumerate}
\end{definition}
\begin{remark}
Let $A$ be an algebra of global dimension $n$, for every module $M$, the functors $\Ext^{n}(M,-)$ and $\Ext^{n}(-,M)$ are right exact Adachi, Enomoto, Tsukamoto had the interesting idea in \cite{aet} to use these functors in order to endowed the module category with what they called an \emp{$\Ext^{-1}$-structure} by setting $\Ext^{-1}(M,N) = \Ext^{n}(M,N)$. They also introduce the notion of $s$-torsion pairs for \emp{extriangulated} categories with an $\Ext^{-1}$-structure as the torsion pairs satisfying $\Ext^{-1}(\T,\F) = 0$. Their setting is much more general than ours, but the two notions coincide in the special case of an algebra of finite global dimension $n$. 
\end{remark}

Let $X\in \modf A$, and $\pi : P \to X\to 0$ a \emp{projective cover} of $X$. The kernel of $\pi$ is denoted by $\Omega_X$ and is called the first \defn{syzygy} of $X$. Dually the cokernel of an \emp{injective envelope}, is denoted by $\Omega^{-1}_{X}$ and called the first \defn{cosyzygy} of $X$. More generally if $\cdots \to P_2 \to P_1 \to P_0 \to X\to 0$ is a \emp{minimal projective resolution} of $X$, the image of $P_{n} \to P_{n-1}$ is called the $n$th syzygy of $X$ and is denoted by $\Omega_X^{n}$. Let us recall the following well-known result (see e.g. \cite[Corollary 2.5.4]{benson1998representations} or \cite[Proposition 3.2.3]{Carlson2003}). 

\begin{lemma}\label{lem:min}
Let $P_\bullet = (P_i,d_i)_{i}$ be a projective resolution of $X$. Then $P_\bullet$ is minimal if and only if the complex $\Hom_A(P_\bullet,S)$ has its differentials equal to zero for every simple module $S$. 
\end{lemma}
\begin{proposition}\label{pro:main}
Let $A$ be an artinian algebra and $(\T,\F)$ a torsion pair in $\modf A$. The following are equivalent:
\begin{enumerate}
\item $\Ext^{kn}_A(\T,\F) = 0$ for every $k\geq 1$. 
\item $\Ext^{n}_{A}(\T,\F) = 0$.
\item The category $\T$ is closed under $n$-syzygies:  if $X\in \T$, then $\Omega^{n}_X \in \T$.
\item The category $\F$ is closed under $n$-cosyzygies: if $X\in \F$, then $\Omega^{-n}_X \in \F$.
\end{enumerate}
\end{proposition}
\begin{proof}
Assume first that $\Ext^{n}_A(\T,\F) = 0$ and let $X \in \T$. We consider $(P_i,d_n)_{i\in \mathbb{N}}$ a minimal projective resolution of $X$. Then $\Omega^n_X$ is the image of $d_{n} : P_{n} \to P_{n-1}$ and we denote by $\pi_n : P_n \to \Omega^{n}_X$ the surjection of $P_n$ onto $\Ima(d_n)$. Since $(\T,\F)$ is a torsion pair for $\modf A$, we have an exact sequence $0\to T \to \Omega^{n}_X \overset{f}{\to} F \to 0$ with $T\in \T$ and $F\in \F$ from which we deduce a morphism of complexes:
\[
\xymatrix
{
P_{n+1}\ar[r]^{d_{n+1}}\ar[d] & \ar[r]\ar[d] P_n \ar[r]^{d_n} \ar@{->>}[d]^{\pi_n} & P_{n-1}\ar@{-->>}[ddl]^{s}\ar[d] \ar[r] & \cdots \ar[r] & P_0\\
0\ar[r]\ar[d] & \Omega^n_X\ar@{->>}[d]^{f}\ar[r] &0\ar[d] \\
0\ar[r] & F \ar[r] & 0
}
\]
Since $\Hom_{K(A)}(P_{\bullet},F[n]) \cong \Hom_{D(A)}(P_\bullet,F[n]) \cong \Hom_{D(A)}(X,F[n]) \cong \Ext^{n}_A(X,F) = 0$, the map $f \circ \pi_n : P_n \to F$ is homotopic to zero. So there is $s : P_{n-1} \to F$ such that $f \circ \pi_n = s \circ d_n$. If $F\neq 0$, then it has a simple quotient $\pi : F\to S\to 0$. The map $0 \neq \pi \circ f \circ \pi_n = \pi \circ s \circ d_n$ is in the image of the differential $\delta : \Hom_A(P_{n-1},S) \to \Hom_A(P_n,S)$ and this contradicts the minimality of the projective resolution by \cref{lem:min}. Hence $F = 0$ and $\Omega_X^{n} \cong T\in \T$. 

Now we assume that $\forall X\in \T$, we have $\Omega_X^{n} \in\T$. By induction we have $\Omega_X^{kn} \in \T$ for every $k\geq 1$. Let $F\in \F$ and $P_{\bullet}$ a minimal projective resolution of $X$. Let $\alpha \in \Hom_{K(A)}(P_\bullet,F[kn])\cong \Ext_A^{kn}(X,F)$: 
\[
\xymatrix
{
P_{kn+1}\ar[r]^{d_{kn+1}}\ar[d] & \ar[r]\ar[d]^{\alpha} P_{kn} \ar[r]^{d_{kn}} & P_{kn-1}\ar[d] \ar[r] & \cdots \ar[r] & P_0\\
0\ar[r]&F \ar[r] &0 
}
\]

Then the composition $P_{kn+1} \to P_{kn} \to F$ is zero, so $\Ima(d_{kn+1}) \subseteq \ker(\alpha)$. Hence $\alpha$ factorizes through $P_{kn}/\Ima(d_{kn+1}) = P_{kn} / \ker(d_{kn}) \cong \Ima(d_{kn}) = \Omega_X^{kn}$. Since $F\in \F$ and $\Omega_X^{kn} \in \T$, we have $\alpha = 0$. By duality, we have $1\Rightarrow 2 \Rightarrow 4 \Rightarrow 1$.

\end{proof}
In the special case $n=1$ we have to following easier characterization. 
\begin{lemma}\label{lem:omega}
Let $A$ be an artinian algebra and $(\T,\F)$ be a torsion pair of $\modf A$. The following are equivalent:
\begin{enumerate}
\item $\Ext^{i}_A(\T,\F) = 0$ for all $i\geq 1$.
\item $\Ext^{1}_A(\T,\F) = 0$.
\item $\T$ and $\F$ are two \emp{Serre subcategories}. 
\item $\T$ is closed under first syzygies: if $X\in \T$, then $\Omega_X \in \T$.
\item $\F$ is closed under first cosyzygies: if $X\in \F$, then $\Omega_X^{-1} \in \F$. 
\end{enumerate}
\end{lemma}
\begin{proof}
The only non trivial implications are $2\Rightarrow 3$ and $3\Rightarrow 4$. For $2\Rightarrow 3$, if $0 \to K \to \T \to C\to 0$ is a short exact sequence we get $K \in \,^{\bot}\Fc = \Tc$ so $\Tc$ is a Serre subcategory. The result for $\Fc$ is dual. For $3\Rightarrow 4$, we can copy the proof of \cite[Proposition 3.21]{aet} or  \cite[Theorem 2.9]{dickson1966torsion}. First we see that $\Tc$ is closed under projective covers and since it is closed under submodules, it is closed under first syzygies. 
\end{proof}
We view the set of $\omega_n$-torsion pairs as a subposet of the lattice of torsion classes. That is:
\[
(\T_1,\F_1) \leq (\T_2,\F_2) \Leftrightarrow T_1\subseteq T_2 \Leftrightarrow \F_2\subseteq \F_1. 
\]

\begin{corollary}
Let $A$ be an artinian algebra. 
\begin{enumerate}
\item The poset of $\omega_n$-torsion pairs is a complete sublattice of the poset of torsion classes of $\modf A$. 
\item If $A$ is an algebra of finite global dimension $n$, viewed as an abelian category with $\Ext^{-1}$-structure given by $\Ext^{-1}(X,Y) = \Ext^{n}_{A}(X,Y)$. Then the poset of $s$-torsion pairs is a complete sublattice of the lattice of torsion pairs.
\item The lattice of $\omega_n$-torsion pairs is a \emp{completely semidistributive} lattice. 
\item If $A$ is \emp{$\tau$-tilting finite}, the lattice of $\omega_n$-torsion pairs is a \emp{congruence uniform} lattice. 
\end{enumerate}
\end{corollary}
\begin{proof}
Let $(\T_i,\F_i)_i$ be a family of $\omega_n$-torsion pairs. Their meet in the lattice of all torsion classes is given by $(\bigcap_i \T_i, \star)$. Since $\T_i$ is closed under $n$th syzygies for every $i$, this is also the case of $\bigcap_i \T_i$. Hence the meet of $\omega_n$-torsion pairs is an $\omega_n$-torsion pair. The join is given by $(\star,\bigcap_i\F_i)$ and $\bigcap_i\F_i$ is closed under $n$th cosyzygies. Hence the $\omega_n$-torsion pairs form a complete sublattice. A (complete) sublattice of a (completely) semidistributive lattice is (completely) semidistributive, and by \cite[Theorem 4.3]{day} a sublattice of a congruence uniform lattice is congruence uniform, so the last statement follows from \cite[Theorem 1.3, Corollary 3.12]{DIRRT}. 
\end{proof}
More generally we can consider \emp{operations} $t$ and $f$ that take objects or morphisms of $\modf A$ and produce new objects. For example $t$ could send a morphism to a kernel, or an object to its first  syzygy etc.  We say that $(\Tc,\Fc)$ is a $(t,f)$-torsion pair if  $\Tc$ is closed under $t$ and $\Fc$ is closed under $f$. When $\Tc$ is closed under $t$ if and only if $\Fc$ is closed under $f$, the set of $(t,f)$-torsion pairs is a \emp{complete sublattice} of the lattice of torsion pairs. Here are other known examples. 
\begin{example}
 A torsion pair $(\T,\F)$ is called:
\begin{enumerate}
\item  \defn{split} if $\Ext^{1}_A(\F,\T) = 0$. By \cite[Proposition 1.7]{assem2006elements} A torsion pair is split if and only if $\T$ is closed under $\tau^{-1}$. This is equivalent to $\F$ being closed under $\tau$. 
\item \defn{hereditary} if $\T$ is closed under submodules. By \cite[Theorem 2.9]{dickson1966torsion}, this is equivalent to $\F$ being closed under injective envelopes.
\item \defn{cohereditary} if $\F$ is closed under quotients. This is equivalent to $\T$ being closed under projective covers. 
\end{enumerate}
\end{example}
\begin{remark}
It is tempting to call the sublattices obtained by imposing algebraic conditions on the torsion pairs \defn{the algebraic sublattices} of the lattice of torsion pairs. In general, there are many more sublattices. For example by the Birkhoff representation theorem any finite distributive lattice with $n$ join-irreducible elements is isomorphic to a sublattice of the lattice of torsion pairs of a semisimple algebra with $n$ simple modules. Due to the semisimplicity of the algebra, there is not much choice for an algebraic condition to impose on the torsion pairs.  The next most simple case, is the case of an equioriented quiver of type $A$. The lattice of torsion pairs is isomorphic to the Tamari lattice and it is an open question to understand its sublattices (see \cite{santocanale2013sublattices} for more details). 
\end{remark}
\begin{corollary}
A torsion pair is an $\omega$-torsion pair if and only if it is both hereditary and cohereditary. 
\end{corollary}
\begin{corollary}
Let $A$ be an artin algebra. The poset of $\omega$-torsion pairs is a \emp{finite distributive} lattice.
\end{corollary}
\begin{proof}
The map $U \mapsto \Filt(U)$ is a bijection between the subsets of the isomorphism classes $\mathcal{S}$ of simple $A$-modules and the Serre subcategories of $\modf A$. It follows that the lattice of hereditary torsion pairs is isomorphic to the boolean lattice $(\mathcal{P}(\mathcal{S}),\subseteq)$. Hence the lattice of $\omega$-torsion pairs is finite and as a sublattice of a boolean lattice, it is distributive.
\end{proof}

\begin{remark}
The lattice of $\omega$-torsion pairs is a sublattice of the lattice of $\omega_n$-torsion pairs. From \cref{pro:main}, we see that the $\omega_n$-torsion pairs are $\omega_m$-torsion pairs when \emp{$n$ divides $m$}.  
\end{remark}

\begin{example}
We consider the quiver $\begin{tikzcd}
	1 & 2
	\arrow["a", shift left=2, from=1-1, to=1-2]
	\arrow["b", shift left=2, from=1-2, to=1-1] \end{tikzcd}$
and the algebra $A = kQ/{<}ab{>}$. It is an algebra of global dimension $2$ with $5$ indecomposable modules: the two simple modules $1$ and $2$ and their projective covers $P_1$ and $P_2$. The projective $P_2$ is also the injective envelope of $2$ and the remaining module is $I_1$ the injective envelope of $1$. There are only $6$ torsion pairs, and the poset of torsion classes is as follows:
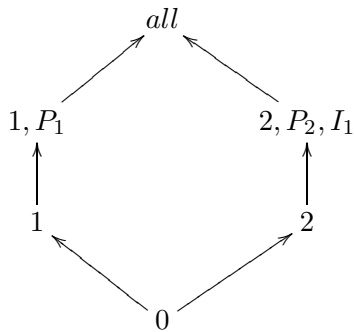
\[
\xymatrix{
 & all &  \\
1,P_1 \ar@{->}[ru] &  & 2,P_2,I_1 \ar@{->}[lu] \\
1 \ar@{->}[u] &  & 2 \ar@{->}[u] \\
 & 0 \ar@{->}[lu] \ar@{->}[ru] & 
}
\]
The hereditary torsion classes are $0, \Filt(1), \Filt(2)$ and $\modf A$. The cohereditary torsion classes are $0, \add(1\oplus P_1), \add(2\oplus I_1\oplus P_2)$ and $\modf A$. The only $\omega$-torsion pairs are $0$ and $\modf A$. The remaining $\omega_2$-torsion pairs are $\Filt(2)$ and $\add(1 \oplus P_1)$. 
\end{example}

\begin{example}
Let $n\in \mathbb{N}$ and $[n]$ a total order with $n$ elements. We denote by $\Int([n],\subseteq)$ the poset of intervals of $[n]$ ordered by containment. The first posets can be found in \cref{fig:dyck}.
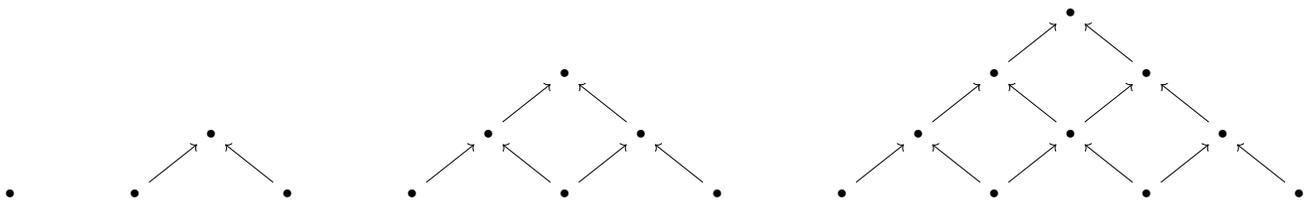
\begin{figure}[h!]
\[
\adjustbox{scale=0.7,center}{
\begin{tikzcd}
	&&&&&&&&&&&&&&& \bullet \\
	&&&&&&&& \bullet &&&&&& \bullet && \bullet \\
	&&& \bullet &&&& \bullet && \bullet &&&& \bullet && \bullet && \bullet \\
	\bullet && \bullet && \bullet && \bullet && \bullet && \bullet && \bullet && \bullet && \bullet && \bullet
	\arrow[from=2-15, to=1-16]
	\arrow[from=2-17, to=1-16]
	\arrow[from=3-8, to=2-9]
	\arrow[from=3-10, to=2-9]
	\arrow[from=3-14, to=2-15]
	\arrow[from=3-16, to=2-15]
	\arrow[from=3-16, to=2-17]
	\arrow[from=3-18, to=2-17]
	\arrow[from=4-3, to=3-4]
	\arrow[from=4-5, to=3-4]
	\arrow[from=4-7, to=3-8]
	\arrow[from=4-9, to=3-8]
	\arrow[from=4-9, to=3-10]
	\arrow[from=4-11, to=3-10]
	\arrow[from=4-13, to=3-14]
	\arrow[from=4-15, to=3-14]
	\arrow[from=4-15, to=3-16]
	\arrow[from=4-17, to=3-16]
	\arrow[from=4-17, to=3-18]
	\arrow[from=4-19, to=3-18]
\end{tikzcd}
}
\]
\caption{First posets of intervals of total orders.}\label{fig:dyck}
\end{figure}
In the case $n=2$, the incidence algebra is hereditary and it has $14$ torsion pairs and only $5$ $\omega$-torsion pairs. For $n=3$ the incidence algebra has global dimension $2$, it has $35$ indecomposable modules and using the computer we found $808$ torsion pairs. The number of $\omega$-torsion pairs is $14$ (see below) and the number of $\omega_2$-torsion pairs is $239$. For $n\geq 4$, there is a representation embedding from the category of modules over a path algebra of extended Dynkin type $\widetilde{D}_4$, to the category of modules over the incidence algebra. Hence the incidence algebra is $\tau$-tilting infinite, but the number of $\omega$-torsion pairs is $42$ (see below). 
\end{example}

As distributive lattice, the poset of $\omega$-torsion pairs is isomorphic to a lattice of order ideals. Here, the description is particularly easy. We denote by $\mathcal{S}$ the set of isomorphism classes of simple modules. When $U \subseteq \mathcal{S}$, we say that $U$ is \defn{closed under successors} if for every $S \in U$ and $S'\in \mathcal{S}$ such that $\Ext^{1}_{A}(S,S') \neq 0$, then $S' \in U$.

\begin{proposition}[Adachi, Enomoto and Tsukamoto]\label{prop:successors}
Let $A = kQ/I$ be an admissible quotient of the path algebra of a finite quiver $Q$. The poset of $\omega$-torsion pairs of $\modf A$ is isomorphic to the poset of subsets of $\mathcal{S}$ which are closed under successors, ordered by inclusion.
\end{proposition}
\begin{proof}
We can copy the proof of \cite[Proposition 3.24]{aet} word for word. If $(\T,\F)$ is an $\omega$-torsion pair, then consider the set $\mathcal{S}(\T)$ of simple modules in $\T$ and we have $\T = \Filt(\mathcal{S}(\T))$ since $\T$ is a Serre subcategory. If $S\in \mathcal{S}(\T)$ and $\Ext^{1}_{A}(S,T) \neq 0$,  then $T$ is a composition factor of the projective cover of $S$. Since $\T$ is closed under projective covers, we get that $T\in \T$. So $\mathcal{S}(\T)$ is closed under successors. 

Conversely if $U\subseteq \mathcal{S}$ is closed under successors, we can consider $\T = \Filt(\{S \ |\ S\in U\})$ and $\F = \Filt(\{S \ |\ S\notin U\})$. By construction $\T$ is a Serre subcategory so it is in particular a torsion class. Hence, it is enough to see that $\T^{\bot} = \F$. Since $U$ is closed under successors, the projective covers of the simple modules of $U$ are in $\T$.  Moreover $\Hom_A(\T,\F) = 0$ since the two categories don't share any simple module. If $M \in \T^{\bot}$, then for every simple $S \in \T$ we have $\Hom(P_S,M) = 0$, where $P_S$ denotes a projective cover of $S$. Hence the composition factors of $M$ are all in $\F$, so $M\in \F$. This proves that $(\T,\F)$ is a torsion pair. Moreover since $\T$ and $\F$ are two Serre subcategories, this is an $\omega$-torsion pair by \cref{lem:omega}.

These two maps are two inverse bijections, and they respect the order relations, so they induce two inverse isomorphisms of posets. 
\end{proof}
As corollaries we have:
\begin{corollary}
Let $A = kQ/I$ be an admissible quotient of $kQ$. Then the lattice of $\omega$-torsion pairs of $A$ is equal to the lattice of $\omega$-torsion pairs of the hereditary algebra $kQ$. 
\end{corollary}
\begin{corollary}\label{cor:w}
Let $(P,\leq)$ be a finite poset. The lattice of $\omega$-torsion pairs of the incidence algebra of $(P,\leq)^{op}$ is isomorphic to the distributive lattice $(\ideal(P),\subseteq)$. 
\end{corollary}

We conclude this note by giving two short proofs of \cref{thm:2}. 

\begin{proof}[First proof of \cref{thm:2}]

By \cite{thomas_tamari}, the Tamari lattice $\Tam_{n+1}$ is isomorphic to the lattice of torsion pairs of an equioriented quiver of type $A_n$. By \cite[Corollary 3.12]{DIRRT} this lattice is \emp{congruence uniform}. Hence, its lattice of congruences is isomorphic to the lattice of ideals of $(P,\leq_f)$ where $P$ is the set of join-irreducible elements and $\leq_f$ is the forcing ordering. By \cite[Theorem 1.4]{DIRRT} the join-irreducible of the lattice of torsion pairs are the bricks in the module category of $\kk A_n$, and they are in bijection with the intervals $[i,j]$ such that $1\leq i \leq j \leq n$. By \cite[Theorem 4.23, Corollary 4.27]{DIRRT}, we have $S_1 \leq_f S_2$ if and only if $S_2$ is a subfactor of $S_1$. In terms of intervals, we have $[i,j] \leq_f [k,l]$ if and only if $[k,l] \subseteq [i,j]$. Hence the lattice of congruences is isomorphic to $\ideal((\Int([n],\subseteq)^{op})$ which is isomorphic to the lattice of Dyck paths of length $n+1$.
\end{proof}
\begin{proof}[Second proof of \cref{thm:2}]
By \cite[Theorem 25]{roitzheim2022n}, the Tamari lattice $\Tam_n$ is isomorphic to the lattice of transfer systems on a total order with $n$ elements. Hence by \cite[Theorem 1.12]{luo2024lattice}, its congruence lattice is isomorphic to $\ideal((\operatorname{Rel}^*(L),\subseteq)^{op})$ where $\operatorname{Rel}^*(L)$ is the set of intervals $[a,b]\in \Int(L)$ with $a\neq b$. The poset  $(\operatorname{Rel}^*(L),\subseteq)$ is clearly isomorphic to $(\Int([n-1],\subseteq)$ and the result follows. 
\end{proof}
  \bibliographystyle{plain}
\bibliography{biblio}
 \end{document}